\documentclass[11pt]{amsart}
\usepackage[T1]{fontenc}
\usepackage{amsmath}
\usepackage{mathrsfs}
\usepackage{amssymb}
\usepackage{framed}
\usepackage{graphicx}

\newtheorem{theorem}{Theorem}[section]
\newtheorem{lemma}[theorem]{Lemma}

\newtheorem{definition}[theorem]{Definition}

\newtheorem{claim}[theorem]{Claim}
\newtheorem{observation}[theorem]{Observation}

\newtheorem{remark}[theorem]{Remark}

\newtheorem{proposition}[theorem]{Proposition}

\newcommand{\im}{\operatorname{Im}}
\newcommand{\re}{\operatorname{Re}}

\newcommand{\Rn}{\mathbb{R}^n}
\newcommand{\R}{\mathbb{R}}

\newcommand{\Cn}{\mathbb{C}^n}
\newcommand{\C}{\mathbb{C}}

\newcommand{\Z}{\mathbb{Z}}

\newcommand{\N}{\mathbb{N}}

\numberwithin{equation}{section}

\newcommand{\abs}[1]{\left|#1\right|}

 \def\HollowBox #1#2{{\dimen0=#1 \advance\dimen0 by -#2
       \dimen1=#1 \advance\dimen1 by #2
        \vrule height #1 depth #2 width #2
        \vrule height 0pt depth #2 width #1
        \llap{\vrule height #1 depth -\dimen0 width \dimen1}%
       \hskip -#2
       \vrule height #1 depth #2 width #2}}
 
\begin{document}

\title[A note on boundary uniqueness]{A note on a conjecture concerning boundary uniqueness}

\author{Abtin Daghighi and Steven G.\ Krantz}

\subjclass[2000]{Primary 
32H12; 
35A02; 
32A10 
}


\keywords{Unique continuation, boundary uniqueness}

\begin{abstract} \rm
We consider the following conjecture (from Huang, et al):
Let $\Delta^+$ denote the upper half disc in $\C$ and let
$\gamma = ( - 1, 1)$ (viewed as an interval in the real axis in $\C$).
Assume that $F$ is a holomorphic function on $\Delta^+$ with continuous extension up to $\gamma$ 
such that $F$ maps $\gamma$ into $\{|\im z|\leq C|\re z|\},$ for some positive $C.$
If $F$ vanishes to infinite order at $0$ then $F$ vanishes identically.
\\
\\
We show that given the conditions of the conjecture, either $F\equiv 0$ or
there is a sequence in $\Delta^+$, converging to $0,$ along which $\im F/\re F$ (defined where $\re F\neq 0$) is unbounded.
\end{abstract}

\maketitle

\section{Introduction} 
The following is a result of Alinhac et al\ \cite{alin1}. 
\begin{theorem}[Alinhac et al\ \cite{alin1}, p.635]\label{alin1thm}  \sl
Let $W\subset\C$ be an open neighborhood of $0$, let $W^+:=W\cap\{\im \zeta >0\}$,
and let $A\subset \Cn$ be a totally real $C^2$-smooth submanifold. Let $F\in {\mathcal O}(W^+)$ and continuous up to the boundary 
such that $F$ maps $W\cap \{ \im \zeta=0\}$ into $A$. If $F$ vanishes to infinite order 
(definition treated in Section 2) at the origin then $F$ vanishes identically in the connected component of the origin in $\overline{W^+}.$
\end{theorem}
There is a related result due to Lakner \cite{lakner} (where it is pointed out that $f(\zeta)=\exp(-e^{i\pi/4}/\sqrt{\zeta})$
is holomorphic on $W^+$ and extends $C^{\infty}$-smoothly to $\overline{W^+},$ yet vanishes to infinite order at $0$).
\begin{theorem}[Lakner \cite{lakner}]\label{laknerthm}	 \sl
Let $\Delta\subset\C$ be the unit disc, let $\Delta^+:=\Delta\cap\{\im \zeta >0\},$
and let $A\subset \C$ be a double cone with vertex at $0$ in the sense that
$A=\{0\}\cup \{\zeta=re^{i\theta},r\in \R,\theta\in [a,b] \},$ for real numbers $a, b$ with $a-b < \pi$. 
Let $F\in {\mathcal O}(\Delta^+)$ and continuous up to the boundary 
such that $F$ maps $\Delta\cap \{ \im \zeta=0\}$ into $A$. If $F|_{\Delta\cap \{ \im \zeta=0\}}$ 
has an isolated zero at $0$ then $F$ does not
vanish to infinite order\footnote{Lakner \cite{lakner} defines vanishing to infinite order at $0$, by $F(\zeta)= O(\zeta^k),$ for all $k\in \N$.}
at $0$. In particular, if $F$ vanishes to infinite order at $0$, then
there is a sequence in $\Delta\cap \{ \im \zeta=0\},$ converging to $0,$ and consisting of zeros of $F.$ 
\end{theorem}

These results were followed up and refined see e.g.\ Baouendi \& Rothschild \cite{bar1}, \cite{baroth}, and Huang et al\ \cite{huang}.  
We mention the following:
\begin{theorem}[Baouendi \& Rothschild \cite{bar1}]\label{bar1thm}  \sl
If $f(\zeta)$ is a holomorphic function 
in a domain of the upper half plane with $0$ on the boundary, continuous up to the boundary,
vanishing to infinite order at $0$, and $\re f(x)\geq 0$ (with $x:= \re \zeta$), then $f$ must vanish identically.
\end{theorem}
\begin{theorem}[Huang et al\ \cite{huang}]\label{huangthm}   \sl
If $f = u+iv$ is holomorphic in $H^+:=\{\zeta\in \C:\hbox{Im} \, \zeta>0\},$
and continuous up to $(-1, 1)\subset \partial H^+$, such that 
$\abs{v(t)} \leq \abs{u(t)}$
for $t \in (-1, 1),$ and if $f$ vanishes to infinite order at $0$ (in the sense that $f(\zeta)=O(\abs{\zeta}^k),$ $H^+\ni \zeta\to 0$, $\forall k\in \N$), 
then $f \equiv 0$. 
\end{theorem}
We also mention the following related result for harmonic functions.
\begin{theorem}[Baouendi \& Rothschild \cite{baroth}, Theorem 1, p.249]\label{barot}	\sl
Let $U\subset \Rn$ be an open neighborhood of $x_0 \in \partial B_0(1)$, where $B_0(1)$ denotes the unit ball centered at 0 
in Euclidean space $\Rn$.
Let $v$ be a harmonic function in $U\cap B_0 (1)$ and continuous on
$\overline{U\cap B_0 (1)}$ (where it is assumed that 
$U\cap B_0 (1)$ is connected).
Assume that, for each positive integer $N,$ the function $s\mapsto \abs{v(s)}/\abs{s-x_0}^{N}$ is integrable on $V=U\cap \partial B_0 (1).$
Then there exists a sequence of real numbers $\{a_j\}_{j\in \N}$ such that, for every positive integer $N$,
the following holds\footnote{The original theorem also ensures that there exists a $D \geq 0$ such that
$\left | a_j - \left [ (-1)^j M_j/(n\omega_n) \right ] \left ( \int_{V} v(s)/(\abs{s-x_0}^{n+2j}) \right ) d\sigma(s) \right | \leq D^{j+1}, \quad j\in \N$
where $M_j=(1/j!)(n/2)(n/2+1)\cdots (n/2+j-1),$ $\omega_n$ is the volume of $B$, and $d\sigma(s)$ is the surface measure on $\partial B_0(1).$} 
true:
$$
\frac{t^{(n-1)/2}}{1+t}v(tx_0)=\sum_{j=0}^N a_j\left(\frac{1-t}{\sqrt{t}}\right)^{2j+1} +O \biggl ( (1-t)^{2N+3} \biggr ),\quad t\to 1^-. \eqno (1.1)
$$
\end{theorem}
Here the following definition is used.
\begin{definition}[Baouendi \& Rothschild \cite{baroth}]\label{barotdefvtio}  \rm
Let $U\subset\Rn$ be a neighborhood of some point $x_0\in \{x\in\Rn :\abs{x} =1 \}$ and set $\Omega =U\cap B_0 (1) .$
A continuous function, $v,$ is said to be {\it vanishing to infinite order at $x_0$}
if 
$$
\lim_{\Omega \ni x\to x_0 } \frac{v(x)}{\abs{x-x_0}^N} = 0   \eqno (1.2)
$$
for all $N > 0$.
The function $v$ is said to {\it vanish to infinite order in the normal direction at $x_0$} if
$$
\lim_{(0,1)\ni t\to \mathbf{1}} \frac{v(tx_0)}{\abs{1-t}^N}=0   \eqno (1.3)
$$
\end{definition}
It has been an open question whether it is possible to replace, 
in Theorem \ref{huangthm}, the inequality $\abs{v(t)} \leq \abs{u(t)}$
for $t \in (-1, 1),$ by an inequality of the form appearing in the Theorem of Lakner \cite{lakner} (Theorem \ref{laknerthm}), i.e.,
$\abs{v(t)} \leq C\abs{u(t)},$ for some $C>0.$ This was conjectured in Huang et al \cite{huang}.
{\em The purpose of this note is to investigate the conjecture}.
\begin{remark}	  \rm
Regarding the property of vanishing to infinite order, we point out the following.
Let $\omega\subset \C$ be a domain and let $f$ be a function continuous on $\overline{\omega}$. Assume that $f$ vanishes to infinite order at a point $p\in \partial \omega$,
in the sense of Theorem \ref{huangthm}, i.e.,
$f(\zeta)=O(\abs{\zeta-p}^k),$ $\overline{\omega}\ni \zeta\to p$, $\forall k\in \N.$
Note that, for any $p\in \omega,$ we have (sufficiently near $p$), $\abs{f(\zeta)}\cdot \abs{\zeta-p}^{-(k+1)} \leq C_{k+1}\Rightarrow \abs{f(\zeta)}\cdot \abs{\zeta-p}^{-k}\leq C_{k+1}\abs{\zeta-p}$;
thus, letting $\zeta\to p,$ we see that 
$$ 
\lim_{\overline{\omega}\ni \zeta\to p} \frac{f(\zeta)}{\abs{\zeta-p}^k}=0,\quad k\in \N,  \eqno (1.4)
$$
(where the case $k=0$ is due to the fact that $\abs{f(\zeta)}\leq C_1\abs{\zeta-p}\to 0$ as $\zeta\to p$).
\end{remark}

\section{Statement and proof of our main result}
\begin{proposition}[Main result]\label{pr}  
Let $\Delta^+$ denote the upper half disc in $\C$ and let
$\gamma=(-1,1)$ (viewed as an interval on the real axis in $\C$).
Assume that $F$ is a holomorphic function on $\Delta^+$ with continuous extension up to $\gamma,$ 
such that $F$ maps $\gamma$ into $\{|\im \zeta|\leq C|\re \zeta|\},$ for some positive $C.$ If $F$ vanishes to infinite order at $0$ (in the same sense as in Theorem \ref{huangthm}) then, either there is a sequence in $\Delta^+$, converging to $0,$ along which $\mbox{Im} F/\mbox{Re} F$ 
(defined where $\mbox{Re} F\neq 0$) is unbounded,
or $F$ vanishes identically. 
\end{proposition}
\begin{proof}
Assume that there exists an $F\not \equiv 0,$ such that $F$ satisfies all other conditions in the statement of the proposition.
(By the result of Huang et al\ \cite{huang}, we may suppose that $1 < C < \infty$.)
Note that $F$ vanishing to infinite order at $0$ implies that, for each $j\in \N$, there is a $C_j> 0$
such that, near $0$,
$\abs{F(\zeta)}\abs{\zeta}^{-j} \leq C_j,$ which in turn implies that, near $0$, we have
$$
\frac{\abs{\re F(\zeta)}}{\abs{\zeta}^{j}}\leq \frac{\abs{F(\zeta)}}{\abs{\zeta}^j} \leq C_j. \eqno (2.1)
$$
Thus the function $\re F$ (and similarly $\im F$) also vanishes to infinite order at $0.$ The function $\re F$ ($\im F$) is harmonic on $\Delta^+$ and
continuous up to $\{y=0\}\cap \Delta$ (where $y=\im \zeta$). 
\\
\\
The strategy of the proof is to show that Lakner's cone condition (the requirement that $F$ map $\gamma$ into a double cone) can, under the additional condition of vanishing to infinite order at the origin, be transported to an appropriate {\em open} part of 
the the upper half disc.
\\
\\
In the following passage, let $p_0=0,$ denote our reference point, as we shall perform a change of coordinates.
\begin{remark}\label{nyremark} \rm
Let $\widehat{B}\subset\Delta^+,$ be a simply connected domain with boundary $\partial\widehat{B}\ni p_0,$ of class $C^{0,\alpha},$ for some $\alpha>0.$
By the Riemann mapping theorem (the homeomorphic extension to the boundary is well-known; see e.g., Taylor \cite{taylor}, p.342; see also Greene \& Krantz \cite{greenekrantz}, p.389) there exists a biholomorphism
$\Psi_2 :\widehat{B}\to \Delta$ which is $C^{0,\alpha}$ up to the boundary. Similarly let $\Psi_1 : \Delta^+ \to \Delta$ be a biholomorphic map ($C^{0,\alpha}$ up to the boundary). For any two points $p_1,p_2\in \partial \Delta,$ we can find a biholomorphic map $\Psi_3:\Delta\to \Delta$ ($C^{0,\alpha}$ up to the boundary) such that $\Psi_3(p_1)=p_2.$ Setting $p_1=\Psi_1(0)$ and $p_2:=\Psi_2(p_0)$
we obtain a biholomorphism $\Psi :\Delta^+\to  \widehat{B}$
($C^{0,\alpha}$ up to the boundary), such that $0=\Psi^{-1}(p_0),$ by defining
$\Psi:=(\Psi_2^{-1}\circ \Psi_3\circ \Psi_1).$ 
Hence $(F\circ\Psi)\in {\mathcal O}(\Delta^+)\cap C^{0,\alpha}(\overline{\Delta^+}).$
\end{remark}
\begin{claim}\label{nyclaim}  \rm
Let $\widehat{B},\Psi$ be as in Remark \ref{nyremark}.
Let $z$ be a holomorphic coordinate centered at $p_0=0,$ 
and set $\Psi^{-1}(z)=:\zeta.$ Then,
$\lim_{\zeta\to 0} (F\circ \Psi)(\zeta)/\zeta^k = 0 \, , \ \forall k\in \N.$ 
\end{claim}
\begin{proof}
Given the holomorphic coordinate $z$ centered at $p_0,$ 
and $\Psi^{-1}(z)=:\zeta,$ we have
$\zeta=\zeta(z),$ $\zeta(p_0)=0,$ 
and, by the infinite order vanishing of $F,$ at $z=p_0$, 
$$
\lim_{z\to p_0} \frac{\abs{F(z)}}{\abs{z-p_0}^j}
=0,\quad \forall j\in \N \, .	\eqno (2.2)
$$
Because $\Psi$ is of class $C^{0,\alpha}$ up to the boundary, and $\Psi (0)=p_0,$ we have,
for constants $\alpha>0,$ $c>0,$ that $\abs{\Psi (\zeta)-p_0}\leq c\abs{\zeta }^{\alpha}$. Whence, for\footnote{Here $\lceil \ \cdot \ \rceil$ denotes the least upper integer.}
$m:=\lceil \frac{1}{\alpha} \rceil,$ and a constant $c_0=c^m,$
$$
\abs{\Psi (\zeta)-p_0}^m\leq c_0\abs{ \zeta}. \eqno (2.3)
$$
This implies that for any $j\in \N$,  \\ 
$$
\lim_{\zeta\to 0}\frac{\abs{(F\circ \Psi)(\zeta)}}{\abs{\zeta}^{k}} = 
\lim_{\zeta\to 0}\frac{\abs{F(\Psi(\zeta))}}{\abs{\Psi(\zeta)-p_0}^{j}}\cdot \frac{\abs{\Psi(\zeta)-p_0}^j}{\abs{\zeta}^k} \leq
$$
$$
\left( \lim_{z\to p_0} \frac{\abs{F(z)}}{\abs{z-p_0}^{j}}\right)\cdot \left( \lim_{\zeta \to 0} \frac{\abs{\Psi(\zeta)-p_0}^j}{\abs{\zeta}^k} \right). \eqno (2.4)
$$ 
Inserting $j=m\cdot k$, in equation (2.4) we obtain, from equation (2.2) and equation (2.3),
$$
\lim_{\zeta\to 0}\frac{\abs{(F\circ \Psi)(\zeta)}}{\abs{\zeta}^{k}}\leq \left ( \lim_{z\to p_0} \frac{\abs{F(z)}}{\abs{z-p_0}^j}\right)
\cdot  \lim_{\zeta \to 0} \underbrace{\left(\frac{\abs{\Psi(\zeta)-p_0}^m}{\abs{\zeta}}\right)^k }_{\ \leq \ c_0^k} = 0 \, . \eqno (2.5)
$$ 

Hence we have verified that (the following limits as $\zeta\to 0$ exists),  
$\lim_{\zeta\to 0} (F\circ \Psi)(\zeta)/\zeta^k = 0,  \ \forall \  k\in \N.$ 
This completes the proof of Claim \ref{nyclaim}.
\end{proof}

\begin{lemma}\label{lm3}
Assume there exists a constant $K>0$ together with a simply connected domain $\widehat{B}\subseteq\Delta^+,$ with boundary $\partial\widehat{B}\ni 0,$ of class $C^{0,\alpha},$ for some $\alpha>0,$ such that, 
$$
\abs{\im F(p)}\leq K\abs{\re F(p)},\quad \forall p\in \widehat{B} . \eqno (2.6)
$$
Then $F\equiv 0.$
\end{lemma}
\begin{proof}
Let $K>0,$ be a constant together with a simply connected domain $\widehat{B}\subseteq\Delta^+,$ with $C^{0,\alpha}$ boundary and $0\in \partial\widehat{B},$ such that inequality (2.6) holds true. If $F\not\equiv 0$, then, the open mapping theorem implies that,
$$
F(p)\neq 0,\quad \forall p\in \widehat{B}, \eqno (2.7)
$$
indeed, by equation (2.6), the (necessarily open) image of the open $\widehat{B},$ under $F$, does not contain an open neighborhood of $0$.
Applying the same arguments of Remark \ref{nyremark} and Claim \ref{nyclaim}, we can find
a biholomorphism $\Psi:\Delta^+\to \widehat{B}$, which is $C^{0,\alpha}$ up to the boundary, such that $(F\circ\Psi)\in {\mathcal O}(\Delta^+)\cap C^{0,\alpha}(\overline{\Delta^+}),$ and such that $(F\circ \Psi)$ vanishes to infinite order at $0.$
Now $F\circ \Psi(\Delta^+)\subseteq F(\widehat{B})$.  Thus $(F\circ \Psi)$ also has image contained in $\{ \abs{\im \zeta} \leq K\abs{\re \zeta}\},$ implying that
$\abs{\im (F\circ \Psi)} \leq K\abs{\re (F\circ \Psi)}$ on $\Delta^+.$
This implies
$$
((F\circ \Psi)(p)\neq 0)\Rightarrow (\re (F\circ \Psi)(p)\neq 0), \quad \forall p\in \Delta^+. \eqno (2.8)
$$
Hence, $\re (F\circ \Psi)$ is a harmonic function on $\Delta^+$, continuous up to the boundary, vanishing to infinite order at $0$ (by Claim \ref{nyclaim} together with equation (2.1)), and nowhere zero on $\Delta^+$.
Hence $\re (F\circ \Psi)$ is either nonpositive or nonnegative, on $\Delta^+$. 
If $\re (F\circ \Psi)$ is nonnegative on $\Delta^+$, then application of Theorem \ref{bar1thm} gives $\re (F\circ \Psi)\equiv 0.$ This implies $(F\circ \Psi)\equiv$constant (by the open mapping theorem) thus by continuity $(F\circ \Psi)\equiv 0.$
If instead, $\re (F\circ \Psi)$ is nonpositive on $\Delta^+$, we can obtain the same conclusion, by replacing $\re (F\circ \Psi)$ by (the necessarily harmonic function) $-\re (F\circ \Psi)$.
In any case $(F\circ \Psi)\equiv 0$ implies that $F$ vanishes on the open set $\Psi(\Delta^+),$ thus $F\equiv 0.$
\end{proof}

\begin{observation}\label{revisobs} \rm
$F\not\equiv 0,$ implies that $\{p\in \Delta^+\colon \re F(p)\neq 0\}$ is nonempty, open, and dense (since the zero set $\{ \re F=0\}$ then has empty interior).
In fact we also have $\{ \re F|_{\gamma}\neq 0\}\subset \gamma$ is a relatively open, dense subset because the condition
$\abs{\im F(p)}\leq C\abs{\re F(p)} ,$ $\forall p\in \gamma$ implies that $\{ \re F|_{\gamma}=0\}$ has empty (relative) interior\footnote{For example it follows from the Luzin-Privalov theorem \cite{luzin}.} in $\gamma,$ when $F\not\equiv 0$. 
\end{observation}

\begin{claim}\label{nyclaim1}	\sl
If $F\not\equiv 0,$ then, either there is a simply connected subdomain, $\widehat{B},$ of $\Delta^+,$ with $C^{0,\alpha}$ boundary passing through $0,$ which satisfies the conditions of Lemma \ref{lm3} or
there is a sequence in $\Delta^+$, converging to $0,$ along which $\mbox{Im} F/\mbox{Re} F$ (defined where $\mbox{Re} F\neq 0$) is unbounded.
\end{claim}
\begin{proof}
Assume equation (2.6) fails for every simply connected subdomain, $\widehat{B},$ of $\Delta^+,$ with $C^{0,\alpha}$ boundary passing through $0.$
It is then obvious that for any increasing sequence $\{K_j\}_{j\in \Z_+},$ of integers, $\widehat{B},$ contains a sequence, $\{p_j\}_{j\in\Z_+}$, such that,
$$
\abs{\im F(p_j)}> K_j \abs{\re F(p_j)}. \eqno (2.9)  
$$ 
Furthermore, the inequality being strict implies that equation (2.9) remains valid on an open neighborhood, $U_j,$ of $p_j.$ If
$U_j\subset \{ \re F=0\},$ then $\re F\equiv 0,$ so $F$ is constant, and by continuity $F\equiv 0$ in which case equation (2.6) holds trivially, thus 
the conditions of Lemma \ref{lm3} are satisfied for any appropriate $\widehat{B}.$ So we can assume
$U_j\cap \{\re F\neq 0\}\neq \emptyset$, for all $j\in \Z_+,$ in which case we replace $p_j$ with a possibly different, point (which, after renaming, in order to retain our notation) again is denoted $p_j,$ satisfying equation (2.9) and such that $\re F(p_j)\neq 0$ (note that automatically also $\im F(p_j)\neq 0$).
Now this can be repeated after replacing $\widehat{B},$ by $\widehat{B}_k$, defined as a simply connected component of the interior of the connected component of $0,$ in 
$\overline{\widehat{B}} \cap \left\{ \abs{\im z}\leq \frac{1}{k}, |\re z|\leq \frac{1}{k} \right\}.$ Thus we can pick a diagonal sequence, again denoted $\{ p_j\}_{j\in \Z_+},$ such that dist$(p_j,0)\to 0.$
\\
\\
By Observation \ref{revisobs}, the set $\omega:=\{\re F \neq 0\},$ is nonempty, contains a dense open subset of
$\Delta^+,$ and contains a relatively open, relatively dense subset of $\gamma.$
The continuous function $g(z):=\abs{\im F(z)}/\abs{\re F(z)},$ $z\in \gamma\cap \omega,$ has a continuous extension to $\omega,$
which we shall denote by $G(z):=\abs{\im F(z)}/\abs{\re F(z)},$ $z\in \omega.$
By what we have already done, we know that, choosing $K_j=j,$ we obtain that, $G(p_j)$ is unbounded along $p_j$ and $p_j\to \infty,$ so we are done. On the other hand, assume no such sequence exists. Then by what we have done, we can conclude that the assumption that equation (2.6) fails for all possible choices of $\widehat{B}$, is false.
This completes the proof of Claim \ref{nyclaim1}.
\end{proof}

Lemma \ref{lm3} together with Claim \ref{nyclaim1} complete the proof of Proposition \ref{pr}.
\end{proof}


\bibliographystyle{amsplain}

\vspace*{.25in}

\noindent Department of Mathematics, Mid Sweden University, SE-851 70, Sundsvall, Sweden \\
{\tt abtindaghighi@gmail.com}
\medskip \\

\noindent Department of Mathematics, Washington University in St.
Louis, St. Louis, Missouri 63130, USA  \\
{\tt sk@math.wustl.edu}

\end{document}